\documentclass[11pt]{amsart}

\usepackage{amsmath,amssymb,amsthm}
\usepackage{fullpage}
\usepackage[all]{xy}
\usepackage{color}
\usepackage{versions}

\renewcommand{\bar}[1]{#1\llap{$\overline{\phantom{\rm#1}}$}}

\DeclareMathOperator{\Hom}{Hom}

\DeclareMathOperator{\Per}{Per}\DeclareMathOperator{\PGL}{PGL}

\DeclareMathOperator{\GL}{GL}

\newcommand{\col}{\,{:}\,}
\newcommand{\tth}{^{\operatorname{th}}}
\newcommand{\sst}{^{\operatorname{st}}}
\newcommand{\nnd}{^{\operatorname{nd}}}

\theoremstyle{plain}
\newtheorem{thm}{Theorem}
\newtheorem{lem}[thm]{Lemma}
\newtheorem{prop}[thm]{Proposition}
\newtheorem{cor}[thm]{Corollary}
\theoremstyle{definition}
\newtheorem{defn}[thm]{Definition}
\newtheorem*{conj}{Conjecture}
\newtheorem{exmp}[thm]{Example}

\theoremstyle{remark}
\newtheorem{case}{Case}

\def\Z{\mathbb{Z}}
\def\Q{\mathbb{Q}}
\def\P{\mathbb{P}}
\def\A{\mathbb{A}}
\def\F{\mathbb{F}}

\def\C{\mathbb{C}}

\author[Hutz, Tepper]
{Benjamin Hutz and Michael Tepper}

\subjclass[2010]{37P05, 37P45.}

\keywords{moduli space, dynamical systems, McMullen's theorem, multiplier spectrum}

\address{Department of Mathematical Sciences\\
Florida Institute of Technology\\
150 W. University Blvd\\
Melbourne, Florida 32901
}
\email{bhutz@fit.edu}

\address{Division of Science and Engineering \\
Penn State Abington\\
1600 Woodland Rd
Abington, PA 19001-3918}
\email{mlt16@psu.edu}

\includeversion{code}

\begin{document}
\title{Multiplier Spectra and the Moduli Space of Degree 3 Morphisms on $\P^1$}

\maketitle

\begin{abstract}
    The moduli space of degree $d$ morphisms on $\P^1$ has received much study.  McMullen showed that, except for certain families of Latt\`es maps, there is a finite-to-one correspondence (over $\C$) between classes of morphisms in the moduli space and the multipliers of the periodic points.  For degree $2$ morphisms Milnor (over $\C$) and Silverman (over $\Z$) showed that the correspondence is an isomorphism \cite{Milnor,Silverman9}.  In this article we address two cases with algebraic methods: polynomial maps of any degree and rational maps of degree $3$.
\end{abstract}

\section{Introduction}
    In this article we examine the moduli space of degree $d$ morphisms of $\P^1$ over an alebraically closed field of charactersitic zero.  In particular, we examine the degree of a finite-to-one map from the moduli space to a parameter space described by McMullen \cite{McMullen2}.  We focus on the explicit computation of the number of necessary parameters and the degree of this map for polynomial morphisms of any degree and rational morphisms of degree $3$.  The case $d=2$ was done by Milnor \cite{Milnor}.  We focus on the case $d=3$ for rational maps. There are several articles that discuss the use of fixed point multlipliers to parameterize polynomial maps, \cite{Fujimura, Nishizawa, Sugiyama}. Here, we focus on including information from higher order multipliers.

    Let $\Hom_d$ be the space of degree $d$ endomorphisms of $\P^1$.  Let $\phi \in \Hom_d$, after choosing coordinates for $\P^1$, we may represent the coordinates of $\phi$ as two degree $d$ homogeneous polynomials with no common zeros.  We may consider $\Hom_d \subset \P^{2d+1}$ by identifying a morphism $\phi$ with its tuple of coefficients. There is a natural action of $\PGL_{2}$ by conjugation on $\Hom_d$, which extends to $\P^{2d+1}$, and we get a moduli space $M_d= \Hom_d/\PGL_{2}$ \cite{Levy,Silverman9}.  The moduli space $M_d$, and its generalization to $\P^N$, has received much study, for example, see \cite{DeMarco, Levy, Manes3, McMullen2, Milnor, Silverman9}.  Milnor \cite{Milnor} gave an isomorphism $M_2 \cong \A^2$ over $\C$.  Silverman in \cite{Silverman9} extended this to $\Z$ in addition to showing that $M_d$ is an affine integral scheme over $\Z$. However, for $d>2$ less is known about the structure of $M_d$.  McMullen \cite{McMullen2} showed that there is a finite-to-one correspondence (over $\C$) between classes of morphisms in the moduli space and certain conjugation invariants called multipliers.  It is known that by considering just the fixed point multipliers for polynomial maps of degree $d$, this correspondence is $(d-2)!$-to-$1$, \cite{Fujimura}. It is this correspondence of McMullen that we study in this article.  We next supply the necessary definitions to state the correspondence precisely.
    \begin{defn}
       Let $\phi \in \Hom_d$ and $\Per_n(\phi) = \{P \in \P^1 \col \phi^n(P) = P\}$ be the set of periodic points of period $n$ for $\phi$. For $P \in \Per_n(\phi)$, $\phi^n$ induces a map on the cotangent space of $\P^1$ at $P$ to itself. The induced map is an element of $\GL_1$ (a scalar) called the \emph{multiplier} at $P$ and is denoted $\lambda_P(\phi)$.

      The \emph{$n$-multiplier spectrum} is the set $\Lambda_n = \{\lambda_P(\phi) \col P \in \Per_n(\phi)\}$, where the multipliers are taken with appropriate multiplicity.

       Define $\sigma_{n,i}$ for $1 \leq i \leq d^n+1$ as the $i\tth$ elementary symmetric function on the $n$-multiplier spectrum.  We denote the $(d^n+1)$-tuple $\boldsymbol\sigma_n = (\sigma_{n,1},\ldots,\sigma_{n,d^n+1})$.
    \end{defn}
    It is an easy application of the chain rule to show that $\Lambda_n$ is invariant under conjugation, as an unordered set.  The multipliers depend algebraically, but not rationally, on the coefficients on $\phi$.  However, the $\sigma_{n,i}$ are actually rational functions in the coefficients of $\phi$ \cite{Milnor, Silverman9}.  Furthermore, the $\sigma_{n,i}$ are regular functions on $M_d$ \cite{Silverman9}.
    \begin{defn}
        Define the map
        \begin{align*}
            \tau_{d,n}:M_d &\to \A^k\\
            [\phi] &\mapsto (\boldsymbol\sigma_1,\boldsymbol\sigma_2,\ldots,\boldsymbol\sigma_{n}).
        \end{align*}
        We define the \emph{degree of $\tau_{d,n}$} as the number of points in $\tau_{d,n}^{-1}(P)$ for a generic point $P$ in $\tau_{d,n}(M_d)$.  From \cite[Theorem 4.54]{Silverman10} the degree of $\tau_{d,n}$ stabilizes as $n \to \infty$ and we write $\deg(\tau_d)$ for this value.
    \end{defn}
    Specifically, McMullen \cite[Corollary 2.3]{McMullen2} showed that there are finitely many conjugacy classes $[\phi] \in \Hom_d$ with a given set of multiplier spectra if $\phi$ is not a flexible Latt\`es map.  This is typically stated as the following theorem.
    \begin{thm}\cite[Corollary 2.3]{McMullen2}
        Fix $d \geq 2$.  For $n$ sufficiently large, $\tau_{d,n}$ is finite-to-one on $M_d(\C)$ except for certain families of Latt\`es maps.
    \end{thm}
    Two maps which have the same set of multipliers are called \emph{isospectral}. All flexible Latt\`es maps (integer multiplication on elliptic curves) are isospectral, since they all have the same multiplier spectra.  Additionally, Latt\`es maps associated to non-isomorphic elliptic curves are non-conjugate. In other words, the regular maps $\sigma_{n,i}$ are constant on the family of flexible Latt\`es maps.  Furthermore, one can use rigid Latt\`es maps to show that $\deg(\tau_d) \geq C_{\epsilon}d^{\frac{1}{2}-\epsilon}$ for some constant $C$.  In particular, for squarefree $d$, $\deg(\tau_d) \geq h$ where $h$ is the class number of $\Q(\sqrt{-d})$ \cite[\S 6.6]{Silverman10}.  In this article we address two cases on $\P^1$: polynomial maps of any degree and rational maps of degree $3$.

\section{Results}

    In Section \ref{sect_polynomials} we examine the locus of polynomial maps in $M_d$.  A \emph{polynomial map} is a map with a totally ramified fixed point.  We denote $\tilde{\tau}_{d,n}$ as the restriction of $\tau_{d,n}$ to the space of polynomial maps of degree $d$.

    \begin{thm} \label{thm_poly_2_mult}
        The maps $\tilde{\tau}_{k,2}$ for $k \in \{2,3,4,5\}$ are one-to-one.
    \end{thm}
    For $k=2$ and $3$ this was already known, since $\tilde{\tau}_{k,1}$ is one-to-one for $k=2$ and $3$. However, $\tilde{\tau}_{k,1}$ is not one-to-one for $k >2$.  In particular its degree $(k-2)!$ \cite{Fujimura}. Our methods work in theory for any $k$, but $k=5$ is the current feasible computational limit. We conjecture the general behavior.
    	\begin{conj}
		McMullen's mulitplier map is one-to-one when restricted to polynomials using the $1\sst$ and $2\nnd$ multiplier spectra.
	\end{conj}

    In Section \ref{sect_degree3} we consider rational maps and examine $M_3$.  Our main result is computing the degree of the correspondence $\tau_{3,2}$ (the number of isospectral maps up to conjugation equivalence) using Groebner bases.
    \begin{thm} \label{thm_32}
        The value $\deg(\tau_{3,2}) = 12$.
    \end{thm}
    While this value represents an upper bound on $\deg(\tau_3)$ it is not necessarily equal to $\deg(\tau_3)$. See Example \ref{exmp_poly} for an example where additional multiplier information causes the degree to decrease.  While our methods theoretically would allow us to compute $\deg(\tau_3)$, the computations with additional multiplier information did not finish.

    An interesting open question is to determine what information uniquely specifies a class in $M_d$.  Towards this end, in Section \ref{sec_mult_fixed} we restrict to the case of a generic map, where generic is the open dense subset of $M_d$ of maps with distinct fixed points.
    \begin{thm} \label{thm_3}
        If $\phi$ has distinct fixed points, then the following correspondence is one-to-one:
        \begin{align*}
            \tau_{d,1}^{+}:M_d \to \A^{2d+1}\\
              [\phi] \mapsto (\boldsymbol\sigma_1,\Per_1(\phi)).
        \end{align*}
    \end{thm}

\section{Polynomials} \label{sect_polynomials}
    We first examine polynomial maps, $\phi$ with a totally ramified fixed point.
    \begin{defn}
        We denote $P_d \subset M_d$ as the moduli space of degree $d$ polynomial maps.

        Define the restriction of $\tau_{d,n}$ to the space of polynomial maps as
    		\begin{align*}
        		\tilde{\tau}_{d,n}:P_d  &\to \A^{k}\\
        		[\phi] &\mapsto (\boldsymbol\sigma_1,\ldots,\boldsymbol\sigma_n).
    		\end{align*}
    \end{defn}
    Any given polynomial map may be conjugated to the form
    \begin{equation} \label{eq_poly_normal_form}
        z^d + a_{2}z^{d-2} + \cdots + a_{d}.
    \end{equation}
    Note however that this is not a normal form since multiple such polynomials may be conjugate.

\subsection{Multipliers of fixed points}
    \begin{prop} \label{prop_sigma_relation}
        Let $[\phi] \in M_d$. The symmetric functions $\boldsymbol\sigma_1$ satisfy
        \begin{equation*}
            (-1)^{d+2}\sigma_{1,d+1} + (-1)^{d-1}\sigma_{1,d-1} + (-1)^{d-2} 2\sigma_{1,d-2} + \cdots  - (d-1)\sigma_{1,1} +d=0.
        \end{equation*}
    \end{prop}
    \begin{proof}
        Label the multipliers of the fixed points as $\{\lambda_0,\ldots,\lambda_{d}\}$.
        Assume first that the fixed points are distinct.  Distinct fixed points implies the multipliers are all different from $1$ and we may apply the relation \cite[Theorem 1.14]{Silverman10}
        \begin{equation*}
                \sum_{i=0}^d \frac{1}{1-\lambda_i} = 1.
        \end{equation*}
        We clear denominators and write both sides in terms of symmetric functions to get
        \begin{equation*}
            \sum_{i=0}^d (-1)^i (d+1-i)\sigma_{1,i} = \sum_{i=0}^{d+1} (-1)^i \sigma_{1,i}.
        \end{equation*}
        Combining the two sides we have the desired result when $\phi$ has distinct fixed points.

        The set of $\phi$ with $1 \in \Lambda_1$ is a Zariski closed set, so the $\phi$ with distinct fixed points are dense in $\Hom_d$.  Thus the function,
        \begin{equation*}
            \sum_{i=0}^d (-1)^i (d+1-i)\sigma_{1,i} - \sum_{i=0}^{d+1} (-1)^i \sigma_{1,i},
        \end{equation*}
        is identically zero.
    \end{proof}
    \begin{cor} \label{cor_sigma_relation}
        For $[\phi] \in P_d$ we have
        \begin{equation*}
            (-1)^{d-1}\sigma_{1,d-1} + (-1)^{d-2} 2\sigma_{1,d-2} + \cdots  - (d-1)\sigma_{1,1} +d=0.
        \end{equation*}
    \end{cor}
    \begin{proof}
        The fixed point at $\infty$ has $\lambda = 0$ and hence $\sigma_{1,d+1}=\prod_i \lambda_i =0$. Now use the formula from Proposition \ref{prop_sigma_relation}.
    \end{proof}

    It is known that specifying $\boldsymbol\sigma_1 \in \tilde{\tau}_{d,1}(\A^{d+1})$ determines a polynomial map in $P_d \subset M_d$ up to finitely many choices, in particular, up to $(d-2)!$ choices. We repeat these results using our algebraic methods for $d=2,\ldots,5$ as they will be needed to prove Theorem \ref{thm_poly_2_mult}.

    \begin{lem} \label{lem_poly_fixed_eq}
        Let $\phi \in P_d$ with affine fixed points $\{z_1,\ldots,z_d\}$. Each $\lambda_i \in \Lambda_1$ such that $\lambda_i \neq 1$ determines an equation of the form
        \begin{equation*}
            F_i(z_1,\ldots,z_{d})= \lambda_i-1
        \end{equation*}
        where the $F_i$ are homogeneous polynomials of degree $d-1$.
    \end{lem}
    \begin{proof}
        For a polynomial map $\phi(z)$, which we may assume is monic, we write
        \begin{equation*}
            \phi(z)-z = \prod_{i=1}^d (z-z_i).
        \end{equation*}
        If $\lambda_i=1$, then $z_i$ is a multiple root of the above equation and upon taking derivatives we get a tautology.  For $\lambda_i \neq 1$ we compute
        \begin{equation*}
            \phi'(z_i) - 1 = \lambda_i-1 = \prod_{j=1, j \neq i}^d (z_i-z_j).
        \end{equation*}
        Thus, we get
        \begin{equation*}
            F_i(z_1,\ldots,z_d) = \prod_{j=1, j \neq i}^d (z_i-z_j) = \lambda_i-1,
        \end{equation*}
        where $F_i(z_1,\ldots,z_d)$ is a homogeneous equation of degree $d-1$.
    \end{proof}
	Each equation in Lemma \ref{lem_poly_fixed_eq} defines a hypersurface and for all but finitely many choices the $\lambda_i$, these hypersurfaces intersect properly. Additionally, we have a $(d+1)$-st hypersurface $\sigma_{d+1} =0$, since $\phi$ is a polynomial map. Thus, we recover the result that McMullen's map is already is finite-to-one from $\Lambda_1$. However, this does not provide the degree. We can compute the degree using Groebner basis for the system of equations from Lemma \ref{lem_poly_fixed_eq} for small $d$. Recall from \cite{Fujimura, Sugiyama}
        \begin{equation*}
            \deg(\tilde{\tau}_{d,1}) = \begin{cases}
              1 & d=2\\
              1 & d=3\\
              2 & d=4\\
              6 & d=5.
            \end{cases}
        \end{equation*}

    This says that $\Lambda_1$ specifies a polynomial up to finitely many choices.  The next example shows that $\Lambda_2$ can further distinguish between polynomials.
    \begin{exmp} \label{exmp_poly}
        Consider the map $\tilde{\tau}_{4,1}(P_4) \to \A^5$ and the fiber $$\tilde{\tau}_{4,1}^{-1}(-1724,-1163982,74470803,4530821869,0).$$  In other words $\Lambda_1 = \{-2243,-59,0,67,511\}$.  There are two polynomials (up to conjugation) in this inverse image
        \begin{align*}
            f(z) &= z^4 - 77z^2 + 217z - 140\\
            g(z) &= z^4 - 721/8z^2 + 217z + 165025/256.
        \end{align*}
        However,
        \begin{equation*}
            \tilde{\tau}_{4,2}(f) \neq \tilde{\tau}_{4,2}(g).
        \end{equation*}

\begin{code}
\begin{verbatim}
Pari/gp code to compute the 2-muliplier spectrum of f and g

f(x)= x^4 - 77*x^2 + 217*x - 140
g(x)= x^4 - 721/8*x^2 + 217*x + 165025/256
Rf=polroots(f(f(x))-x);
Rg=polroots(g(g(x))-x);
{
  for(i=1,16,
   print(subst(deriv(f(f(x)),x),x,Rf[i]));
  );
print("---------------");
  for(i=1,16,
   print(subst(deriv(g(g(x)),x),x,Rg[i]));
  );
}
\end{verbatim}
\end{code}
    \end{exmp}
    \begin{prop} \label{prop_42}
        $\deg(\tilde{\tau_4}) = \deg(\tilde{\tau}_{4,2})=1$.
    \end{prop}
    \begin{proof}
        We use a Groebner basis calculation in Magma \cite{magma}.  We add to the fixed point equations, the equations for a single $2$-periodic point
        \begin{align*}
            \phi(\phi(\beta)) = \beta \qquad \text{and} \qquad (\phi^2)'(\beta)=\lambda_{\beta}.
        \end{align*}
        From the fixed point equations there were 6 distinct choices of the set of fixed points  with only 2 up to conjugation.  For each set of fixed points there are $16$ possible $2$-periodic points, and, hence, there are at least $96$ points on this variety.  Specializing to $\Lambda_1 = \{-4,6,-3,-1/3,0\}$ and using Magma we have a zero-dimensional scheme in the coordinates $(\beta,\lambda_{\beta})$ which is reduced with $96$ distinct points.  Working modulo $13$, where the scheme is still reduced we determine the $96$ points over the algebraic closure of $\F_{13}$ and that there are $2$ distinct possibilities for $\Lambda_2$.  Thus, there is one polynomial map associated to $\{\Lambda_1,\Lambda_2\}$. Since these points are all multiplicity one, there will remain $2$ distinct maps under perturbation of $\Lambda_1$.

\begin{code}
\begin{verbatim}
//Magma code:

R<z1,z2,z3,B,l4>:=AffineSpace(GF(13),5);

function f(x)
  return(x + x^4 + (-z1^2 + (-z2 - z3)*z1 + (-z2^2 - z3*z2 - z3^2))*x^2 +
  ((z2 + z3)*z1^2 + (z2^2 + 2*z3*z2 + z3^2)*z1 + (z3*z2^2 + z3^2*z2))*x +
  (-z3*z2*z1^2 + (-z3*z2^2 - z3^2*z2)*z1));
end function;

//helper function
function W(i,S,T)
  if S[i] eq 0 then
    return(0);
  else
    return(S[i]/T[i]);
  end if;
end function;

//find the points over the splitting field
l1:=-4;
l2:=6;
l3:=-3;
f1:=Evaluate(Derivative(f(B),B),B,z1) - l1-1;
f2:=Evaluate(Derivative(f(B),B),B,z2) - l2-1;
f3:=Evaluate(Derivative(f(B),B),B,z3) - l3-1;
f4:=f(f(B))-B;
f5:=Derivative(f(f(B)),B)- l4;

C:=Scheme(R,[f1,f2,f3,f4,f5]);
IsReduced(C);
Degree(C);
Pc:=PointsOverSplittingField(C);

//check how many sets of points are distinct.
Z:=[];
for i:=1 to 6 do
A:={};
good:=1;
  for j:=1 to 16 do
    A:=Include(A,Pc[(i-1)*16+j]);
  end for;
  for k:=1 to #Z do
    if Z[k] eq A then good:=0; end if;
  end for;
  if good eq 1 then  Z:=Append(Z,A); end if;
end for;
#Z;

//check to see which of the distinct sets of points differ by a
//third root of unity
numdistinct:=0;
notdone:=[1..#Z];
for a:=1 to #Z do
  if (a in notdone) eq true then
    numdistinct:=numdistinct+1;
    for b:=a+1 to #Z do
      match:=0;
      r:=[];
      T:=[[P[1]^3,P[2]^3,P[3]^3,P[4]^3,P[5]^3] : P in Z[a]];
      S:=[[Q[1]^3,Q[2]^3,Q[3]^3,Q[4]^3,Q[5]^3] : Q in Z[b]];
      TT:=[[P[1],P[2],P[3],P[4],P[5]] : P in Z[a]];
      SS:=[[Q[1],Q[2],Q[3],Q[4],Q[5]] : Q in Z[b]];
      match:=1;
      for j:=1 to 16 do
        found:=0;i:=1;
        while (found eq 0) and (i le 16) do
          if S[j] eq T[i] then
            if r eq [] then
               r := [W(1,TT[i],SS[j]),W(2,TT[i],SS[j]),W(3,TT[i],SS[j]),
                    W(4,TT[i],SS[j]),W(5,TT[i],SS[j])];
               found:=1;
            else
              good:=1;
              for k:=1 to 5 do
                if SS[j][k] ne 0 then
                  if r[k] ne W(k,TT[i],SS[j]) then
                    good:=0;
                  end if;
                end if;
              end for;
              if good eq 1 then found:=1; end if;
            end if;
          end if;
          i:=i+1;
        end while;
        if found ne 1 then match:=0; end if;
      end for;
      if match eq 1 then
        notdone:=Exclude(notdone,b);
      end if;
    end for;
  end if;
end for;
print numdistinct;
\end{verbatim}
\end{code}
    \end{proof}
    \begin{prop} \label{prop_52}
        $\deg(\tilde{\tau_5}) = \deg(\tilde{\tau}_{5,2})=1$.
    \end{prop}
    \begin{proof}
        We again use a Groebner basis calculation in Magma \cite{magma}.  We add to the fixed point equations, the equations for a single $2$-periodic point
        \begin{align*}
            \phi(\phi(\beta)) = \beta \qquad \text{and} \qquad (\phi^2)'(\beta)=\lambda_{\beta}.
        \end{align*}
        From the fixed point equations there were $24$ distinct choices of the set of fixed points with 6 up to conjugation.  For each set of fixed points there are $25$ possible $2$-periodic points, and, hence, there are at least $600$ points on this variety.  Specializing to $\Lambda_1 = \{6,-4,5,3,3/7,0\}$ and using Magma we have a zero-dimensional scheme in the coordinates $(\beta,\lambda_{\beta})$.  Working modulo $29$  the scheme is reduced with $600$ distinct points over the algebraic closure of $\F_{29}$.  There are $6$ distinct possibilities for $\Lambda_2$.  Thus, there is one polynomial map associated to $\{\Lambda_1,\Lambda_2\}$. Since these points are all multiplicity one, there will remain $6$ distinct maps under perturbation of $\Lambda_1$.

\begin{code}
\begin{verbatim}
//Magma Code:

R<z1,z2,z3,z4,B,l5>:=AffineSpace(GF(29),6);

function f(x)
  return(x+x^5 + (-z1^2 + (-z2 + (-z3 - z4))*z1 + (-z2^2 + (-z3 - z4)*z2 +
  (-z3^2 - z4*z3 - z4^2)))*x^3 + ((z2 + (z3 + z4))*z1^2 + (z2^2 + (2*z3 + 2*z4)*z2 +
  (z3^2 + 2*z4*z3 + z4^2))*z1+ ((z3 + z4)*z2^2 + (z3^2 + 2*z4*z3 + z4^2)*z2 +
  (z4*z3^2 + z4^2*z3)))*x^2 + (((-z3 - z4)*z2 - z4*z3)*z1^2 + ((-z3 - z4)*z2^2 +
  (-z3^2 - 3*z4*z3 - z4^2)*z2 + (-z4*z3^2 - z4^2*z3))*z1 +(-z4*z3*z2^2 +
  (-z4*z3^2 - z4^2*z3)*z2))*x + (z4*z3*z2*z1^2 +
  (z4*z3*z2^2 + (z4*z3^2+ z4^2*z3)*z2)*z1));
end function;

//helper function
function W(i,S,T)
  if S[i] eq 0 then
    return(0);
  else
    return(S[i]/T[i]);
  end if;
end function;

//find the points over the splitting field
l1:=6;
l2:=-4;
l3:=5;
l4:=3;
f1:=Evaluate(Derivative(f(B),B),B,z1) - l1-1;
f2:=Evaluate(Derivative(f(B),B),B,z2) - l2-1;
f3:=Evaluate(Derivative(f(B),B),B,z3) - l3-1;
f4:=Evaluate(Derivative(f(B),B),B,z4) - l4-1;

f5:=f(f(B))-B;
f6:=Derivative(f(B),B)- l5;

C:=Scheme(R,[f1,f2,f3,f4,f5,f6]);
IsReduced(C);
Degree(C);
Pc:=PointsOverSplittingField(C);

//check how many sets of points are distinct.
Z:=[];
for i:=1 to 24 do
A:={};
good:=1;
  for j:=1 to 25 do
    A:=Include(A,Pc[(i-1)*25+j]);
  end for;
  for k:=1 to #Z do
    if Z[k] eq A then good:=0; end if;
  end for;
  if good eq 1 then  Z:=Append(Z,A); end if;
end for;
#Z;

//check to see which of the distinct sets of points differ by a
//fourth root of unity
numdistinct:=0;
notdone:=[1..#Z];
for a:=1 to #Z do
  if (a in notdone) eq true then
    numdistinct:=numdistinct+1;
    for b:=a+1 to #Z do
      match:=0;
      r:=[];
      T:=[[P[1]^4,P[2]^4,P[3]^4,P[4]^4,P[5]^4,P[6]^4] : P in Z[a]];
      S:=[[Q[1]^4,Q[2]^4,Q[3]^4,Q[4]^4,Q[5]^4,Q[6]^4] : Q in Z[b]];
      TT:=[[P[1],P[2],P[3],P[4],P[5],P[6]] : P in Z[a]];
      SS:=[[Q[1],Q[2],Q[3],Q[4],Q[5],Q[6]] : Q in Z[b]];
      match:=1;
      for j:=1 to 25 do
        found:=0;i:=1;
        while (found eq 0) and (i le 25) do
          if S[j] eq T[i] then
            if r eq [] then
               r := [W(1,TT[i],SS[j]),W(2,TT[i],SS[j]),W(3,TT[i],SS[j]),
                    W(4,TT[i],SS[j]),W(5,TT[i],SS[j]),W(6,TT[i],SS[j])];
               found:=1;
            else
              good:=1;
              for k:=1 to 6 do
                if SS[j][k] ne 0 then
                  if r[k] ne W(k,TT[i],SS[j]) then
                    good:=0;
                  end if;
                end if;
              end for;
              if good eq 1 then found:=1; end if;
            end if;
          end if;
          i:=i+1;
        end while;
        if found ne 1 then match:=0; end if;
      end for;
      if match eq 1 then
        notdone:=Exclude(notdone,b);
      end if;
    end for;
  end if;
end for;
print numdistinct;
\end{verbatim}
\end{code}
    \end{proof}

	Combining the two previous propositions proves Theorem \ref{thm_poly_2_mult}.
	
	Note that we are computing degree as the number of points in a generic inverse image. The following, an example commuincated by Adam Epstein, is one such special fiber for $d=4$ where the correspondence is not one-to-one.
	\begin{exmp}
		Define
		\begin{equation*}
			f_{a,b} = (z^2+a)^2 + b^2
		\end{equation*}
		Then
		\begin{equation*}
			\Lambda_n(f_{a,b}) = \Lambda_n(f_{b,a}) \qquad \text{for all } n.
		\end{equation*}
	\end{exmp}

\subsection{Explicit $P_3^1$}
    Milnor \cite{Milnor} gave an explicit normal form for classes in $M_2$ in terms of $\Lambda_1$.  We give a similar description for $P_3$. In particular, we give an explicit description of the fiber $\tilde{\tau}_{3,1}^{-1}(\boldsymbol\sigma_1) \subset P_3$.
    \begin{prop}
        Let $[\phi] \in  \tilde{\tau}_{3,1}^{-1}(\boldsymbol\sigma_1)$ and write $[\phi]$ in the form $\phi(z)= z^3 + az+b$. If $\phi$ has $3$ distinct affine fixed points with multipliers $\lambda_1$, $\lambda_2$, and $\lambda_3$.  Then,
        \begin{align*}
            a&=-\frac{(\lambda_1^2 + (\lambda_2 - 6)\lambda_1 + (\lambda_2^2 - 6\lambda_2 + 9))}{(3\lambda_1 + (3\lambda_2 - 6))}\\
            27b^2 &= \sigma_{1,3} - \sigma_{1,2} a + \sigma_{1,1} a^2 - a^3.
        \end{align*}

        If $\phi$ has $2$ distinct affine fixed points with $\lambda$ the multiplier that is not $1$. Then,
        \begin{align*}
            a &=1-\frac{\lambda -1}{3}\\
            b &= -2\left(\pm \frac{\lambda -1}{9}\right)^{3/2}.
        \end{align*}

        If $\phi$ has a single affine fixed point, then
        \begin{equation*}
            a=b=0.
        \end{equation*}
    \end{prop}
    \begin{proof}
        Specifying $\boldsymbol\sigma_1$ specifies the $1$-multiplier spectrum $\Lambda_1 = \{\lambda_1,\lambda_2,\lambda_3,0\}$.
        \begin{case}[$3$ distinct affine fixed points]
            For the three affine fixed points, we know
            \[ z_i = \pm \sqrt{\frac{\lambda_i-a}{3}}.\]
            Additionally, we have
            \begin{equation*}
                z^3 + (a-1)z + b = \prod_{i=1}^3 (z-z_i)
            \end{equation*}
            where
            \begin{align*}
                b &= -\prod_{i=1}^3 z_i = - \prod_{i=1}^3 \pm \sqrt{\frac{\lambda_i-a}{3}}\\
                (a-1) &= z_1z_2 + z_1z_3 + z_2z_3\\
                0 &= z_1+z_2+z_3 = \sum_{i=1}^3 \pm \sqrt{\frac{\lambda_i-a}{3}}.
            \end{align*}
            Solving for $a$ and $b$, we obtain
            \begin{align*}
                a&=-\frac{(\lambda_1^2 + (\lambda_2 - 6)\lambda_1 + (\lambda_2^2 - 6\lambda_2 + 9))}{(3\lambda_1 + (3\lambda_2 - 6))}\\
                27b^2 &= \sigma_{1,3} - \sigma_{1,2} a + \sigma_{1,1} a^2 - a^3.
            \end{align*}
        \end{case}
        \begin{case}[$2$ distinct affine fixed points]
            Let $\lambda_3=\lambda$ be the multiplier that is not equal to $1$, then since $z_1+z_2+z_3=0$ and $z_1=z_2$ we know
            \begin{equation*}
                z_3 = -2z_1
            \end{equation*}
            and
            \[\phi(z) -z = z^3 + (a-1)z + b = (z-z_1)^2(z+2z_1).\]
            Thus, we have
            \begin{align*}
                a &= 1-3z_1^2\\
                b&=2z_1^3\\
                \lambda&=3z_3^2 + a = 12z_1^2 + a = 9z_1^2+1.
            \end{align*}
            Thus, we solve
            \[ z_1 = \pm \sqrt{\frac{\lambda -1}{9}}\]
            and
            \begin{align*}
                a &=1-\frac{\lambda -1}{3}\\
                b &= \pm 2\left(\frac{\lambda -1}{9}\right)^{3/2}.
            \end{align*}
        \end{case}
        \begin{case}[$1$ distinct affine fixed point]
            If $\lambda_i = 1$ for $i=1,2,3$, then $z_1=z_2=z_3$ for the three affine fixed points since they must all have multiplicity at least $2$.  Since $z_1+z_2+z_3=0$, we must have $z_i=0$ and, hence, $\phi(z) = z^3$.
        \end{case}
    \end{proof}
    Similar to the computation that $\phi_c(z) = z^2+c$ is the family $\boldsymbol \sigma_1 = (2,4c,0)$ in $M_2^1$, we can compute the image of degree $3$ polynomials.
    \begin{cor}
        We may write $[\phi] \in P_3$ as
        \begin{equation*}
            \phi_{a,b}(z) = z^3 + az + b.
        \end{equation*}
        up to the sign of $b$.  In particular,
        \begin{equation*}
            [\phi_{a,b}] = [\phi_{a,-b}].
        \end{equation*}
    \end{cor}
    \begin{prop}
        The image of $[\phi_{a,b}]$ under $\tau_{3,1}$ is given by
        \begin{equation*}
            \tau_{3,1}([\phi_{a,b}]) = \boldsymbol \sigma_1 = (6-3a,9-6a,9a-12a^2 + 4a^3 + 27b^2,0)
        \end{equation*}
    \end{prop}
    \begin{proof}
        Direct computation was performed with Mathematica \cite{Mathematica7}.
\begin{code}

\begin{verbatim}
Mathematica code:

f[x_] := x^3 + a*x + b;
df[x_] := D[f[x],x]
T := Solve[{f[x] == x}, {x}]
L1 := {};
For[i = 1, i < 4, i++, AppendTo[L1, df[T[[i, 1, 2]]]]]
For[i = 1, i < 4, i++,
Print[Simplify[Expand[SymmetricPolynomial[i, L1]]]]]
\end{verbatim}
\end{code}
\end{proof}

\section{Rational Maps} \label{sect_degree3}
	We now prove Theorem \ref{thm_32}.
    \begin{proof}
        The proof proceeds in two steps.  First we determine a zero-dimensional variety whose points give the coefficients of the map.  Then we determine the number of points on this variety.

        Recall that $\deg(\tau_3) = \#(\tau_3^{-1}(P))$ for a generic point $P$.  Generically, there are $4$ distinct fixed points  ($\lambda_i = 1$ is a closed condition) and at least one 2-periodic point which is not also a fixed point ($\lambda_i \neq -1$).

        We dehomogenize and write $\phi$ as a rational map denoted as $\bar{\phi}(z)$
        We label the coefficients of the general such map as
        \begin{equation*}
            \bar{\phi}(z) = \frac{a_1z^3 + a_2z^2 + a_3z + a_4}{b_1z^3 + b_2z^2 + b_3z + b_4}.
        \end{equation*}
        Under a $\PGL_2$ transformation we may move two of the fixed points to $0$ and $\infty$ to obtain
        \begin{equation*}
            \bar{\phi}(z) = \frac{a_1z^3 + a_2z^2 + a_3z}{b_2z^2 + b_3z + b_4}.
        \end{equation*}
        Computing the multipliers we have
        \begin{equation*}
            \lambda_0 b_4 = a_3 \qquad \lambda_{\infty} a_1 = b_2.
        \end{equation*}
        We are left to determine the coefficients $\{a_1,a_2,b_3,b_4\}$.  A $\PGL_2$ transformation allows us to move a third fixed point to $1$.  Then we have
        \begin{equation*}
            a_2 = (b_2 + b_3 + b_4) - a_1 - a_3.
        \end{equation*}
        Looking at the multiplier at $z=1$ we can solve for
        \begin{equation*}
            b_3=\frac{(1-\lambda_1\lambda_{\infty})a_1 + (2-\lambda_0-\lambda_1)b_4}{\lambda_1-1},
        \end{equation*}
        since $\lambda_1 \neq 1$.  Letting $\alpha$ be the fourth and last fixed point.  Then we can solve for
        \begin{equation*}
            b_4 = \frac{(\alpha a_1 \lambda_{\infty} - \alpha a_1)}{1-\lambda_0}.
        \end{equation*}
        This provides all the coefficients of $\overline{\phi}(z)$ in terms of $\{\lambda_0,\lambda_1,\lambda_{\infty},\alpha\}$ except for $a_1$.  To account for $a_1$ we may either consider $\phi(z)$ as a map on $\P^1$ and set $a_1=1$ or, equivalently, notice that $a_1$ is a factor of each of the coefficients and cancels in $\overline{\phi}(z)$.  Also, note that $\lambda_{\alpha}$ is uniquely determined by the relation \cite[Theorem 1.14]{Silverman10}
        \begin{equation*}
            \frac{1}{1-\lambda_1} + \frac{1}{1-\lambda_0} + \frac{1}{1-\lambda_{\infty}} + \frac{1}{1-\lambda_{\alpha}} = 1.
        \end{equation*}

        Let $\beta \not\in \{0,1,\infty,\alpha\}$ be an exact 2-periodic point. We have new equations
        \begin{align}
            \bar{\phi}^2(\beta) &= \beta \label{eq3}\\
            \lambda_{\beta} = (\bar{\phi}^2)'(\beta) &= \bar{\phi}'(\beta)\bar{\phi}'(\bar{\phi}(\beta)). \notag
        \end{align}
        Given $\lambda_{\beta}$, the system (\ref{eq3}) has 2 variables $\{\alpha,\beta\}$ and 2 equations.  Except for possibly finitely many choices of $\lambda_{\beta}$, this system defines a $0$-dimensional variety in coordinates $(\alpha,\beta)$.  Thus, for a generic point $P$, the fiber $\tau_{3,2}^{-1}(P)$ is finite-to-one.

        We now determine $\deg(\tau_{3,2})$.  Computations were done in Magma \cite{magma}.   A generic Groebner basis computation does not finish.  Choosing a particular specialization (values of the multipliers) and computing the degree of the reduced subscheme, we find that it has $18$ distinct points in $\P^2$.  We will show there remain 18 distinct points under perturbation of the multipliers.

        We find $6$ points where the map is not generic, where $\alpha=0$ or $1$ is a fixed point with multiplicity greater than one.  In coordinates $(\alpha,\beta,z)$, where $z$ is the homogenizing variable, the 6 points are
        \begin{align*}
            P_1&=(1,0,0)\\
            P_2&=(0,0,1)\\
            P_3&=(1,1,1)\\
            P_4&=(0,\frac{-\lambda_1 - \lambda_{\infty} + 2}{1-\lambda_1},1)\\
            P_5&=(1,\frac{\lambda_{\infty} - 1}{1-\lambda_0},1)\\
            P_6&=(1,-\frac{\lambda_0\lambda_1\lambda_{\infty} - \lambda_0\lambda_1 - \lambda_{\infty} + 1}{\lambda_0\lambda_1 - \lambda_0 - \lambda_1 + 1},0).
        \end{align*}
        By computing the determinant of the Jacobian matrix for $P_4$, $P_5$, and $P_6$ we see that they have generic multiplicity of at least $2$.  We next compute that $P_1$, $P_2$, and $P_3$ each have generic multiplicity $42$.  The method is to pick a specialization where this occurs, and then show that the multiplicities remain constant under perturbation of the multipliers.  Magma is able to compute the multiplicities with 3 of the 4 multipliers fixed, demonstrating that the generic multiplicities are each $42$.  Since the total number of points of intersection (with multiplicities) by B\'ezout's theorem is $144$, we see that the other twelve points must have multiplicity $1$ and $P_4$, $P_5$, and $P_6$ have multiplicity exactly $2$.  Thus, these twelve points also remain distinct under perturbation of the multipliers.

        The points $P_1,\ldots,P_6$ do satisfy the necessary equations, but they correspond to non-generic maps, where $\alpha = 0$ or $1$, causing the fixed points to not be distinct.  It is easy to check that these six points are all of the possibilities for $\alpha, \beta \in \{0,1\}$.  Similarly, if $(\alpha, \beta)$ and $(\alpha, \beta')$ corresponded to the same map, then we would have $\lambda_{\beta} = \lambda_{\beta'}$ which is a non-generic situation.  Thus, the remaining twelve points are inverse images under $\tau_{3,2}$ and the degree of $\tau_{3,2}$ is $12$.

\begin{code}
\begin{verbatim}
\\Pari/gp code
\\defining the map f(x) in pari to get the systems of equations
la=1+1/(1/(1-l8) + 1/(1-l1) + 1/(1-l0)-1);
b2=l8*a1;
a3=l0*b4; \\multiplier of 0
a2=(b2 + b3 + b4) - a1 - a3; \\1 is fixed
b3=((1-l1*l8)*a1 + (2-l0-l1)*b4)/(l1-1); \\multiplier of 1
b4= (a*a1*l8 - a*a1)/(1-l0)
\\multiplier of a is determined uniquely.
f(x) =  (a1*x^3 + a2*x^2 + a3*x)/(b2*x^2 + b3*x + b4)

//---------------------------------------------------------
//use the definition of f(x) from pari to determine the scheme
//Magma code:

S<l0,l1,l8,lB>:=FunctionField(Rationals(),4);
R<a,B,z>:=ProjectiveSpace(S,2);

function h(P,d)
  Q:=0;
   for i:=0 to d do
     for j:=0 to d-i do
       Q:=Q+ Term(Term(P,a,i),B,j)*z^(d-i-j);
     end for;
   end for;
return(Q);
end function;

function f(x,y)
  return((((l0 - 1)*l1 + (-l0 + 1))*x^3 + ((a*l0*l1 + (-l0 + (-a + 1)))*l8 +
  (((-a - 1)*l0 +1)*l1 + (2*l0 + (a - 2))))*x^2*y + ((-a*l0*l1 + a*l0)*l8 +
  (a*l0*l1 - a*l0))*x*y^2));
end function;
function g(x,y)
 return((((l0 - 1)*l1 + (-l0 + 1))*l8*x^2*y + (((-l0 + (a + 1))*l1 +
 (a*l0 - 2*a))*l8 + (-a*l1 + ((-a + 1)*l0 +(2*a - 1))))*x*y^2 +
 ((-a*l1 + a)*l8 + (a*l1 - a))*y^3));
end function;

f1:=f(f(B,1),g(B,1));
g1:=g(f(B,1),g(B,1));

F1:=f1-B*g1;
F2:=g1*Derivative(f1,B) - f1*Derivative(g1,B) - lB*g1*g1;

G1:=h(F1,9);
G2:=h(F2,16);

C:=Scheme(R,[G1,G2]);

//----------------------------------
//Determining the 6 generic high multiplicity points

Factorization(GCD(Evaluate(G1,a,0),Evaluate(G2,a,0)));
Factorization(GCD(Evaluate(G1,B,0),Evaluate(G2,B,0)));
Factorization(GCD(Evaluate(Evaluate(G1,a,0),z,1),Evaluate(Evaluate(G2,a,0),z,1)));
Factorization(GCD(Evaluate(Evaluate(G1,a,0),z,0),Evaluate(Evaluate(G2,a,0),z,0)));
Factorization(GCD(Evaluate(G1,z,0),Evaluate(G2,z,0)));
Factorization(GCD(Evaluate(Evaluate(G1,a,1),z,0),Evaluate(Evaluate(G2,a,1),z,0)));
Factorization(GCD(Evaluate(Evaluate(G1,a,1),z,1),Evaluate(Evaluate(G2,a,1),z,1)));
Factorization(GCD(Evaluate(Evaluate(G1,B,1),z,1),Evaluate(Evaluate(G2,B,1),z,1)));
//---------------------------------------------

P1:=C![1,0,0];
P2:=C![0,0,1];
P3:=C![1,1,1];
P4:=C![0,1/(1-l1)*(-l1 - l8 + 2),1];
P5:=C![1,1/(1-l0)*(l8 - 1),1];
P6:=C![1,-(l0*l1*l8 - l0*l1 - l8 + 1)/(l0*l1 - l0 - l1 + 1),0];

//-----------------
//Compute the Jacobian determinant to show that P4,P5,P6 all have
//multiplicity at least 2 generically.

Evaluate(Evaluate(Evaluate(Derivative(F1,a),a,P4[1]),B,P4[2]),z,P4[3])*
Evaluate(Evaluate(Evaluate(Derivative(F2,B),a,P4[1]),B,P4[2]),z,P4[3])-
Evaluate(Evaluate(Evaluate(Derivative(F1,B),a,P4[1]),B,P4[2]),z,P4[3])*
Evaluate(Evaluate(Evaluate(Derivative(F2,a),a,P4[1]),B,P4[2]),z,P4[3]);

Evaluate(Evaluate(Evaluate(Derivative(F1,a),a,P5[1]),B,P5[2]),z,P5[3])*
Evaluate(Evaluate(Evaluate(Derivative(F2,B),a,P5[1]),B,P5[2]),z,P5[3])-
Evaluate(Evaluate(Evaluate(Derivative(F1,B),a,P5[1]),B,P5[2]),z,P5[3])*
Evaluate(Evaluate(Evaluate(Derivative(F2,a),a,P5[1]),B,P5[2]),z,P5[3]);

Evaluate(Evaluate(Evaluate(Derivative(F1,B),a,P6[1]),B,P6[2]),z,P6[3])*
Evaluate(Evaluate(Evaluate(Derivative(F2,z),a,P6[1]),B,P6[2]),z,P6[3])-
Evaluate(Evaluate(Evaluate(Derivative(F1,z),a,P6[1]),B,P6[2]),z,P6[3])*
Evaluate(Evaluate(Evaluate(Derivative(F2,B),a,P6[1]),B,P6[2]),z,P6[3]);


//---------------------------------------
l0:=3;
l1:=2;
l8:=4;
lB:=-5;

R<a,B,z>:=ProjectiveSpace(GF(101),2);

function h(P,d)
  Q:=0;
   for i:=0 to d do
     for j:=0 to d-i do
       Q:=Q+ Term(Term(P,a,i),B,j)*z^(d-i-j);
     end for;
   end for;
return(Q);
end function;

function f(x,y)
  return((((l0 - 1)*l1 + (-l0 + 1))*x^3 + ((a*l0*l1 + (-l0 + (-a + 1)))*l8 +
  (((-a - 1)*l0 +1)*l1 + (2*l0 + (a - 2))))*x^2*y + ((-a*l0*l1 + a*l0)*l8 +
  (a*l0*l1 - a*l0))*x*y^2));
end function;
function g(x,y)
 return((((l0 - 1)*l1 + (-l0 + 1))*l8*x^2*y + (((-l0 + (a + 1))*l1 +
 (a*l0 - 2*a))*l8 + (-a*l1 + ((-a + 1)*l0 +(2*a - 1))))*x*y^2 +
 ((-a*l1 + a)*l8 + (a*l1 - a))*y^3));
end function;

f1:=f(f(B,1),g(B,1));
g1:=g(f(B,1),g(B,1));

F1:=f1-B*g1;
F2:=g1*Derivative(f1,B) - f1*Derivative(g1,B) - lB*g1*g1;

G1:=h(F1,9);
G2:=h(F2,16);

C:=Scheme(R,[G1,G2]);
D:=ReducedSubscheme(C);
Degree(D);

//-----------------------------------------------
//Compute the multiplicities letting each multiplier be a parameter

//l0:=3;
l1:=2;
l8:=4;
lB:=-5;

S<l0>:=FunctionField(Rationals(),1);
R<a,B,z>:=ProjectiveSpace(S,2);

function h(P,d)
  Q:=0;
   for i:=0 to d do
     for j:=0 to d-i do
       Q:=Q+ Term(Term(P,a,i),B,j)*z^(d-i-j);
     end for;
   end for;
return(Q);
end function;

function f(x,y)
  return((((l0 - 1)*l1 + (-l0 + 1))*x^3 + ((a*l0*l1 + (-l0 + (-a + 1)))*l8 +
  (((-a - 1)*l0 +1)*l1 + (2*l0 + (a - 2))))*x^2*y + ((-a*l0*l1 + a*l0)*l8 +
  (a*l0*l1 - a*l0))*x*y^2));
end function;
function g(x,y)
 return((((l0 - 1)*l1 + (-l0 + 1))*l8*x^2*y + (((-l0 + (a + 1))*l1 +
 (a*l0 - 2*a))*l8 + (-a*l1 + ((-a + 1)*l0 +(2*a - 1))))*x*y^2 +
 ((-a*l1 + a)*l8 + (a*l1 - a))*y^3));
end function;

f1:=f(f(B,1),g(B,1));
g1:=g(f(B,1),g(B,1));

F1:=f1-B*g1;
F2:=g1*Derivative(f1,B) - f1*Derivative(g1,B) - lB*g1*g1;

G1:=h(F1,9);
G2:=h(F2,16);

C:=Scheme(R,[G1,G2]);
P1:=C![1,0,0];
P2:=C![0,0,1];
P3:=C![1,1,1];
P4:=C![0,1/(1-l1)*(-l1 - l8 + 2),1];
P5:=C![1,1/(1-l0)*(l8 - 1),1];
P6:=C![1,-(l0*l1*l8 - l0*l1 - l8 + 1)/(l0*l1 - l0 - l1 + 1),0];

Multiplicity(P1);
Multiplicity(P2);
Multiplicity(P3);
\end{verbatim}
\end{code}
    \end{proof}

\section{Multipliers and Fixed Points} \label{sec_mult_fixed}
     McMullen's theorem tells us that there are only finitely many conjugacy classes with a given set of multiplier spectra.  An interesting question would be to determine what information is necessary to specify a class uniquely in $M_d$.  In the case of distinct fixed points, we prove Theorem \ref{thm_3}.

    \begin{proof}
        We write
        \begin{equation*}
            \phi(z) =  z-  \frac{p(z)}{q(z)}.
        \end{equation*}
        The fixed points are the roots of the polynomial $p(z)$ so it can be specified (up to constant).

        We have
        \begin{equation*}
            \phi'(z) = 1- \frac{p'(z)}{q(z)} - \frac{p(z)q'(z)}{q(z)^2}
        \end{equation*}
        Evaluated at a fixed point is
        \begin{equation*}
            \lambda_i = \phi'(z_i) = 1- \frac{p'(z_i)}{q(z_i)} - 0.
        \end{equation*}
        Thus we get the linear equation
        \begin{equation*}
            (1-\lambda_i)q(z_i) - p'(z_i)=0.
        \end{equation*}
        Note that $\deg(q(z)) = d$ so has $d+1$ coefficients to go along with the $d+1$ fixed points. Finding the coefficients of $q(z)$ when $\lambda_i \neq 1$ is then a question of whether or not the matrix of coefficients of the linear system is invertible.  This matrix is a Vandermonde with each row scaled by the nonzero constant $(1- \lambda_i)$.  Thus, since the fixed points are distinct, the matrix is invertible and there is a unique solution.

        If there is a fixed point at infinity, then the argument is the same except that $\deg(p(z))$ is one less.
    \end{proof}


\subsection{Normal Form for Degree 3}
    We propose the following normal form for a rational map of degree $3$, with distinct fixed points.
    \begin{thm}
        {\footnotesize
        \begin{align*}
            \bar{\phi}(z) &= \frac{((-\lambda_1 + 1)\lambda_0 + (\lambda_1 - 1))z^3 + (((-\alpha\lambda_1 + 1)\lambda_0 + (\alpha - 1))\lambda_{\infty} + (((\alpha + 1)\lambda_1 - 2)\lambda_0 + (-\lambda_1 + (-\alpha + 2))))z^2} {((-\lambda_1 + 1)\lambda_0 + (\lambda_1 - 1))\lambda_{\infty}z^2 + (((\lambda_1 - \alpha)\lambda_0 + ((-\alpha - 1)\lambda_1 + 2\alpha))\lambda_{\infty} + ((\alpha - 1)\lambda_0 + (\alpha\lambda_1 + (-2\alpha + 1))))z}\\
            &\frac{ + ((\alpha\lambda_1 - \alpha)\lambda_0\lambda_{\infty} + (-\alpha\lambda_1 + \alpha)\lambda_0)z)}{ + ((\alpha\lambda_1 - \alpha)\lambda_{\infty} + (-\alpha\lambda_1 + \alpha)))}
        \end{align*}
        }
        Where the fixed points are $\{0,1,\infty,\alpha\}$ with corresponding multipliers $\{\lambda_0,\lambda_1,\lambda_{\infty}, \lambda_{\alpha}\}$ and
        \begin{equation*}
            \lambda_{\alpha} = \frac{1}{\frac{1}{1-\lambda_0} + \frac{1}{1-\lambda_{\infty}} + \frac{1}{1-\lambda_1}-1}+1.
        \end{equation*}
    \end{thm}
    \begin{proof}
        The details are identical to the beginning of the proof of Theorem \ref{thm_32} in Section \ref{sect_degree3}, so we merely state the outline.  The explicit calculation was carried out in Pari \cite{pari232}.

        We conjugate $\phi$ so that $\{0,1,\infty\}$ are fixed points with multipliers $\{\lambda_0,\lambda_1,\lambda_{\infty}\}$.  Labeling the last fixed point as $\alpha$, we have the equation $\phi(\alpha)=\alpha$ allowing us to determine the form of $\phi$ depending only on the choices of $\{\lambda_0,\lambda_1,\lambda_{\infty}, \alpha\}$.

\begin{code}
\begin{verbatim}
//Pari/gp code
la=1-1/(1-(1-l8) - 1/(1-l1) - 1/(1-l0))
b2=l8*a1;
a3=l0*b4; \\multiplier of 0
a2=(b2 + b3 + b4) - a1 - a3; \\1 is fixed
b3=((1-l1*l8)*a1 + (2-l0-l1)*b4)/(l1-1); \\multiplier of 1
f(x) =  (a1*x^3 + a2*x^2 + a3*x)/(b2*x^2 + b3*x + b4)
b4=-(((-a^3 + a^2)*l1 + (a^3 - a^2))*a1*l8 + ((a^3 - a^2)*l1 +
(-a^3 + a^2))*a1)/(((-a^2 + a)*l1 + (a^2 - a))*l0 + ((a^2 - a)*l1 + (-a^2 + a)))

f(x) =  (a1*x^3 + a2*x^2 + a3*x)/(b2*x^2 + b3*x + b4)
\end{verbatim}
\end{code}
    \end{proof}

\providecommand\biburl[1]{\texttt{#1}}

\end{document}